%% file: Stability.tex
\documentclass[a4paper,11pt,fancyhdr]{article}

\usepackage{amsmath,amssymb,graphicx,fancyhdr,indentfirst,%
  undertilde,color,multirow,slashed,mathtools,listings,cite}
\usepackage[english]{babel}
\usepackage[colorlinks=true,urlcolor=black]{hyperref}
\usepackage[makeroom]{cancel}
\usepackage{amsthm}
\usepackage{geometry}
\usepackage{tikz}
\usepackage{graphicx}
\usepackage{textcomp}
\usepackage{tikz-cd}
\usepackage{appendix}
\usepackage{subcaption}
\usetikzlibrary{decorations.markings}
\usepackage[ruled,vlined]{algorithm2e}
\usepackage{mathrsfs}
\usepackage{authblk}
\usetikzlibrary{arrows}

\input{commands}

\title{Euler and Betti curves are stable under Wasserstein deformations of distributions of stochastic processes}

\author[1,2,3]{Daniel Perez\thanks{Email: \texttt{daniel.perez@ens.fr}}}
\affil[1]{\footnotesize D\'epartement de math\'ematiques et applications, \'Ecole normale sup\'erieure, CNRS, PSL University, 75005 Paris, France}
\affil[2]{\footnotesize Laboratoire de math\'ematiques d'Orsay, Universit\'e Paris-Saclay, CNRS, 91405 Orsay, France}
\affil[3]{\footnotesize DataShape, Centre Inria Saclay, 91120 Palaiseau, France}

\date{\today}

\begin{document}
\maketitle

\begin{abstract}
Euler and Betti curves of stochastic processes defined on a $d$-dimensional compact Riemannian manifold which are almost surely in a Sobolev space $W^{n,s}(X,\R)$ (with $d<n$) are stable under perturbations of the distributions of said processes in a Wasserstein metric. Moreover, Wasserstein stability is shown to hold for all $p>\frac{d}{n}$ for persistence diagrams stemming from functions in $W^{n,s}(X,\R)$.
\end{abstract}
\tableofcontents

\section{Introduction}
The study of random fields by Adler and Taylor using geometric and topological methods gave rise to the following celebrated and interesting result.
\begin{theorem}[Adler, Taylor, Theorem 12.4.1 \cite{AdlerTaylor:RandomFields}]
Let $f$ be a centered unit variance Gaussian field on a $d$-dimensional $C^2$-submanifold $X$ of a $C^3$ manifold. Suppose $f$ satisfies the conditions of \cite[Corollary 11.3.5]{AdlerTaylor:RandomFields} (henceforth the GKF conditions), then,
\be
\expect{\chi(\{f \geq x\})} = \sum_{j=0}^d \Lag_j(X) \, \rho_j(x) \, \,,
\ee 
where $\chi$ denotes the Euler characteristic and $\rho_j(x)$ denotes some known functions of $x$ and $\Lag_j$ denote the Lipschitz-Killing curvatures of $X$.
\end{theorem}
In recent years, much work has been done to weaken the Gaussian assumption of the above theorem. A notable step in this direction is the following remarkable result by Bobrowski and Borman.
\begin{theorem}[Bobrowski, Borman, Theorem 4.1, \cite{Bobrowski_2012}]
Let $X$ be a compact d-dimensional stratified space, and let $f : X \to \R^k$ be a $k$-dimensional Gaussian random field satisfying the GKF
conditions. For a piecewise $C^2$-function $G : \R^k \to \R$ and let $g=G \circ f$. Setting $D_u = G^{-1}\,]-\infty,u]$, we have
\be
\expect{\int_X g \,d\chi } = \chi(X)\expect{g} - \sum_{j=1}^d (2\pi)^{j/2} \Lag_j(X) \int_{\R} \mathcal{M}_j(D_u)\;du \,,
\ee
where $\expect{g}$ has constant mean.
\end{theorem}
As pointed out by the authors of the above theorem in their original paper \cite{Bobrowski_2012}, the difficulty in evaluating the expression above lies in computing the Minkowski functionals $\mathcal{M}_j(D_u)$. In this paper, we aim to somewhat compensate this difficulty by providing a stability result regarding the above expectations, and moreover, the distributions of the Euler curves (\textit{i.e.} the curves $x \mapsto \chi(\{f\geq x\}))$. More precisely, our main result is the following.
\begin{theorem}
Let $X$ be a compact $d$-dimensional manifold and let $(\Omega, \mathcal{F},\PP)$ be a probability space. Consider two stochastic processes $f,g: \Omega \to W^{n,s}(X,\R)$ with $d<n$. Then, the Euler curves of $f$ and $g$ are well-defined as elements of $L^1(\R)$ and as distributions, and for every $\frac{d}{n} < \al <1$, there exists a finite constant $C_{X,n,s}$ such that
\be
\norm{\expect{\chi(\{f \geq x\})}-\expect{\chi(\{g \geq x\})}}_{L^1_x(\R)} \leq C_{X,n,s} \left[\norm{f}_{L^1(\Omega,W^{n,s})}^\al+\norm{g}_{L^1(\Omega,W^{n,s})}^\al\right] W_{1,L^\infty}^{1-\al}(f_\sharp \PP, g_\sharp \PP) \,,
\ee
where $W_{1,L^\infty}$ denotes the Wasserstein distance on the space of probability measures on the space $(L^\infty(X,\R) \cap W^{n,s}(X,\R), \norm{\cdot}_{L^\infty})$. Moreover, the same statement holds for all Betti curves.
\end{theorem}

The key will be to adopt the point of view of persistent homology, for which we refer the unfamiliar reader to the now classical references on the subject for an introduction \cite{Chazal:Persistence, Oudot:Persistence}. The relevance of the study of persistence diagrams was already understood by Bobrowski and Borman, and was later also elaborated upon by different authors \cite{Adler_2010, Lebovici_2021}. 

Parallel to this development, the following result was recently shown by Buhovsky \textit{et al.}
\begin{theorem}[Buhovsky, Payette, Polterovich, Polterovich, Shelukhin, Stojisavljevi\'c, \cite{Polterovich_2022}]
\label{thm:Buhovsky}
Let $X$ be a compact Riemannian manifold of dimension $d$ and let $f \in W^{n,s}(X,\R)$ and $n>\frac{d}{s}$. Then, $f$ is continuous and for all $\veps >0$ 
\be
N^\veps_{H_k(X,f)} \leq C_{X,n,s} \norm{f}_{W^{n,s}}^{d/n} \;\veps^{-d/n} + \dim H_k(X) \,,
\ee 
where $N^\veps_{H_k(X,f)}$ denotes the number of bars of length $\geq \veps$ in the barcode of the persistent homology of degree $k$, $H_k(X,f)$.
\end{theorem}
A similar bound was shown in \cite{Perez_2020} for $\al$-H\"older functions and for a wider class of metric spaces. The results of this paper can be seen as direct consequences of the theory developped in \cite{Perez_2020}, but which deserve to be made explicit.

\begin{remark}
Using the Sobolev injection $W^{n,s} \xhookrightarrow{} C^{r+\al}$, for $r+\al = n- \frac{d}{s}$, we can rewrite the exponent as $\frac{d}{r+\al + \frac{d}{s}}$. In particular, when $s= \infty$ and $r=0$, we retrieve the same exponent as in \cite{Perez_2020}.
\end{remark}

\subsection{Glossary, definitions and conventions}
\begin{notation}
We will interchangeably denote the persistent homology in degree $k$ on $X$ with respect to a filtration function $f: X \to \R$ by $H_k(X,f)$ or $\Dgm_k(f)$. We use persistence barcodes, diagrams, modules and their decompositions interchangeably via the usual identifications \cite{Chazal:Persistence}. We reserve the notation $H_*(X)$ for the homology groups of $X$ exclusively.
\end{notation}
\begin{convention}
Unless otherwise specified, we will always consider the filtration by \textbf{super}level sets induced by a function $f: X \to \R$.
\end{convention}
\begin{convention}
Unless otherwise specified, we shall always consider the (infinite) bars of the barcode to be truncated. That is, $b = b \cap [\inf f , \sup f]$ for every $b \in \Dgm_k(f)$.
\end{convention}
\begin{notation}
The length of a bar $b \in \Dgm_k(f)$ is denoted by $\ell(b)$. With this definition, we define 
\be
\Pers_p^p(H_k(X,f)) := \sum_{b \in \Dgm_k(f)} \ell(b)^p \,.
\ee
\end{notation}
\begin{notation}
We equip $\R^2$ with a metric $d_{\R^2,\infty}$ defined by
\be
d_{\R^2,\infty}((x,y),(x',y')) = \max\{\abs{x-x'},\abs{y-y'}\} \,.
\ee
and define the distance from a point $(x,y)$ to the diagonal $\Delta$ as $d_{\R^2,\infty}((x,y),\Delta) = \half \abs{y-x}$. We will use the following notation to design different sets of $\R^2$. For \textbf{super}level set filtrations we denote
\begin{itemize}
  \item $\mathcal{X} = \{(x,y) \in \R^2 \, \vert \, y<x\}$ and $\overline{\mathcal{X}}$ its closure.
  \item $R_x := [x,\infty[ \; \times\;]\!-\!\infty,x]  \,\subset \overline{\mathcal{X}}$.
\end{itemize}
The reader can adjust the statements by modifying the definitions in the following way if he or she is considering \textbf{sub}level set filtrations.
 use the following notation for superlevel and sublevel sets respectively:
\begin{itemize}
  \item $\mathcal{X} = \{(x,y) \in \R^2 \, \vert \, x<y\}$ and $\overline{\mathcal{X}}$ its closure.
  \item $R_x :=\;  ]\!-\!\infty,x] \times [x,\infty[ \,\subset \overline{\mathcal{X}}$.
\end{itemize}
\end{notation}
\begin{definition}
Define $\DD_p$ to be the space of Radon measures on $\overline{\mathcal{X}}$ (called \textbf{persistence measures}) with finite $\Pers_p$ \cite[Equations 3 and 4]{Divol_2019}. We equip $\DD_p$ with the topology of the optimal \textit{partial} transport distance $d_p$ (\cf \cite[Definition 2.1]{Divol_2019}). 
\end{definition}

\section{Stability results}
\subsection{Preliminary results}
\begin{lemma}
\label{lemma:PerspMellin}
\be
\Pers_p^p(H_k(X,f)) = p \int_0^\infty  \veps^{p-1} N^\veps_{H_k(X,f)} \;d\veps \,.
\ee
\end{lemma}
\begin{proof}
The length of a bar $b \in H_k(X,f)$ can be written as  
\be
\ell(b)^p = p\int_0^{\infty} 1_{[0,\ell(b)]}(\veps) \, \veps^{p-1} \;d\veps \,. 
\ee
The result follows by summing over all bars $b$ in the barcode of $H_k(X,f)$ and applying the Fubini theorem.
\end{proof}
A natural corollary of the above facts is that if $f \in W^{n,s}(X,\R)$, then we can bound the $\Pers_p^p$-functional of $H_k(X,f)$ by a universal constant and $\norm{f}_{W^{n,s}}$ as soon as $p > \frac{d}{n}$.
\begin{corollary}
\label{cor:FinitePersp}
Let $f \in W^{n,s}(X,\R)$, then for $p>\frac{d}{n}$, there exists a finite constant $C_{X,n,s}$ such that 
\be
\Pers_p^p(H_k(X,f)) \leq C_{X,n,s}\,  \frac{pn}{d} \, \norm{f}_{W^{n,s}}^{p} + (2\norm{f}_\infty)^p\dim H_k(X) \,.
\ee
\end{corollary} 

\begin{remark}
Corollary \ref{cor:FinitePersp} means
\be
f \in W^{n,s}(X,\R) \implies \Dgm_k(f) \in \bigcap_{p>d/n}\DD_p \,.
\ee
\end{remark}

\begin{remark}
Examining the limiting case where $n \to \infty$ and $s \to \infty$, one can wonder whether $C^\infty$-functions have finite $\Pers_0$ (in other words, a finite number of bars). However, this is not the case as is shown by the example $e^{-1/x^2}\sin(1/x)$ on $[0,1]$. 

The answer of which is the optimal degree of regularity for which $\Pers_0$ is finite is an interesting question. A sufficient condition found through discussions with François Petit is that the functions be subanalytic in an $o$-minimal structure. Once $\Pers_0$ is finite, it follows that $\Pers_p^p$ is a holomorphic function everywhere on $\C$.

On the other side of the story, we can ask whether the $C^\al$ (up to reparametrization by a homeomorphism) condition is necessary to establish that $\Pers_p^p$ is finite for some $p$. This question has been positively answered in \cite{Perez_2020} in the 1D case (in fact, the correct notion of regularity is the $p$-variation, but is equivalent to asking that the function be $C^\al$ for some $\al$, up to reparametrization).  
\end{remark}

One can ask whether the bound $p >\frac{d}{n}$ for the finiteness of $\Pers_p^p$ is sharp within the degree of regularity considered. This question amounts to asking whether the asymptotics provided by theorem \ref{thm:Buhovsky} are in some sense sharp. Indeed, note that lemma \ref{lemma:PerspMellin} entails 
\begin{proposition}
\be
\limsup_{\veps \to 0} \frac{\log N^\veps_{H_k(X,f)}}{\log(1/\veps)} = \inf\{p \, \vert \, \Pers_p(H_k(X,f)) < \infty\} \,.
\ee
\end{proposition}
The sharpness of the asymptotics of $N^\veps_{H_k(X,f)}$ was already partially adressed in \cite[\S 1.5]{Polterovich_2022} and positively answer the question. For $\al$-H\"older functions, this sharpness was already thoroughly adressed in \cite{Perez_2020}, where it was shown that
\be
\inf\{p \, \vert \, \Pers_p(H_k(X,f)) < \infty\} = \frac{d}{\al}  \quad \text{generically in } C^\al \,.
\ee   
This discussion suggests the following conjecture, which might be adressed in a later work.
\begin{conjecture}
Let $X$ be a compact Riemannian manifold of dimension $d$, then 
\be
\inf\{p \, \vert \, \Pers_p(H_k(X,f)) < \infty\} \leq \frac{d}{r+\al}
\ee
and this inequality is saturated generically in the sense of Baire in $C^{r+\al}$.
\end{conjecture}
\begin{remark}
By the work performed in \cite{Perez_2020}, it suffices to exhibit a function $g$ on the cube (or the ball), such that the bound is saturated. We obtain a dense family by perturbating any dense family to make it constant on a small open ball $B \subset X$ and adding $g$ to $f$ on this ball. For $H_0$, such a function may be constructed by considering packings and smoothened versions of distances to point clouds. In higher degrees of homology, exhibiting explicit examples might be more difficult. Experience shows that it might be easier to find a random field which saturates the bound, as done in \cite{Perez_2020}. 
\end{remark}

\subsection{Stability theorems}
Having bounded $\Pers_p^p$, a stability result follows from the standard argument considered in \cite{LipschitzStableLpPers}. 
\begin{theorem}[Wasserstein stability on compact Riemannian manifolds]
Let $X$ be a compact Riemannian manifold of dimension $d$ and let $f, g \in W^{n,s}(X,\R)$. Then, for all $p>q>\frac{d}{n}$, 
\be
d_p^p(\Dgm_k(f),\Dgm_k(g)) \leq C_{X,n,s} \left(\norm{f}_{W^{n,s}}^q \vee \norm{g}_{W^{n,s}}^q\right) \norm{f-g}_{\infty}^{p-q} \,.
\ee
\end{theorem}
\begin{proof}
Following the proof of the stability theorem shown in \cite{LipschitzStableLpPers}, it is sufficient to show that $\Pers_q(\Dgm_k(f))$ and $\Pers_q(\Dgm_k(g))$ can be bounded. This bound is provided by corollary \ref{cor:FinitePersp}, yielding the result.
\end{proof}

Similarly, using \cite{Perez_2020}, this bound immediately entails a stochastic stability result.
\begin{theorem}[Stochastic Wasserstein Stability]
Let $f$ and $g$ be two a.s. $W^{n,s}(X,\R)$ stochastic processes on a $d$-dimensional compact Riemannian manifold $X$, defined on a probability space $(\Omega,\mathcal{F},\PP)$. Then, for any $0 \leq k<d$, every $\frac{d}{n}<q< p < \infty$ and any $r,s \in \;]1,\infty[$ satisfying $\frac{1}{r}+\frac{1}{s}=1$ and $(p-q)s \geq 1$, there exists a constant $C= C_{X,n,s}$ such that 
\begin{align}
W_{p,d_p}((\Dgm_k \circ f)_\sharp\PP,(\Dgm_k \circ g)_\sharp\PP ) &\leq C \left[\expect{\norm{f}_{W^{n,s}}^{qr}}^{\frac{1}{r}}+ \expect{\norm{g}_{W^{n,s}}^{qr}}^{\frac{1}{r}}\right]^{\frac{1}{p}} W_{(p-q)s,\infty}^{1-\frac{q}{p}}(f_\sharp\PP,g_\sharp\PP) \\
&\leq C \left[\expect{\norm{f}_{W^{n,s}}^{qr}}^{\frac{1}{r}}+ \expect{\norm{g}_{W^{n,s}}^{qr}}^{\frac{1}{r}}\right]^{\frac{1}{p}} \norm{f-g}_{L^{(p-q)s}(\Omega,L^\infty(X,\R))}^{1-\frac{q}{p}} \,.
\end{align}
\end{theorem}
\begin{proof}
Here, we follow the proof of \cite[Theorem 5.9]{Perez_2020}. The proof is the same, replacing the $C^\al$-seminorms by the Sobolev norms. 
\end{proof}

\begin{theorem}[Stochastic stability of representations, \cite{Perez_2020}]
Let $\mathcal{B}$ be a Banach space and $\Psi : \mathcal{D}_p \to \mathcal{B}$ be an $\al$-H\"older continuous functional. Let $\PP, \mathbb{Q} \in \mathcal{P}_{\al q}(\DD_p)$, then
\be
\norm{\mathbb{E}_\PP[\Psi] - \mathbb{E}_{\mathbb{Q}}[\Psi]}_{\mathcal{B}} \leq W_{q,\norm{\cdot}_{\mathcal{B}}}(\Psi_\sharp \PP, \Psi_\sharp \mathbb{Q}) \leq \norm{\Psi}_{C^\al(\DD_p, \mathcal{B})} W_{q\al,d_p}^\al(\PP,\mathbb{Q}) \,.
\ee
\end{theorem}

\section{Euler and Betti curves}
Degree-by-degree, we define the so-called Betti and Euler curves. 
\begin{definition}
Let $f \in W^{n,s}(X,\R)$ with $n>\frac{d}{s}$. The \textbf{$k$th degree Betti curve of $f$}, denoted $\beta_k(x,f)$ is defined as
\be
\beta_k(f)(x) := \beta_k(f,x) := \Dgm_k(f)(R_x)\,,
\ee
where $\Dgm_k(f)$ is seen as a persistence measure. The \textbf{Euler curve of $f$} is defined as 
\be
\chi(f)(x) := \chi(f,x) := \sum_{k=0}^d (-1)^k \Dgm_k(f)(R_x)
\ee
\end{definition} 
Considered pointwise, showing a stability result for $\beta_k(f,x)$ is hopeless. This follows from the fact that we may perturb $f$ so that we displace the points in the diagram pushing them over the boundary of $R_x$, thereby introducing a perturbation in $\beta_k(f,x)$ of a potentially infinite number of points. 
\par
The correct point of view is to allow ourselves to consider the $\beta_k(f,x)$ as an element of $L^1(\R)$, or more generally as a distribution.
\begin{proposition}
Let $X$ be a $d$-dimensional Riemannian manifold and let $f \in W^{n,s}(X,\R)$ for $d<n$. Then, for every $k$, $\beta_k(f), \chi(f) \in L^1(\R)$ and thus also in $\DD'(\R)$. Moreover, 
\be
\norm{\beta_k(f)}_{L^1(\R)} = \Pers_1(\Dgm_k(f)) \,.
\ee
\end{proposition}
\begin{remark}
$\beta_k(f)$ can be infinite for some values of $x$. Nonetheless, the degree of regularity considered ensures that we can still make sense of average values of $\beta_k$ around small neighbourhoods of a level $x$. This justifies that we look at the Betti and Euler curves as distributions. 
\end{remark}
\begin{proof}
We do the proof for sublevel sets, noticing that the proof for superlevel sets is completely analogous. Since $\beta_k(f)$ is a positive function, 
\begin{align*}
\norm{\beta_k(f)}_{L^1(\R)} &= \int_\R \beta_k(f,x) \;dx = \int_\R \Dgm_k(f)(R_x) \;dx = \int_\R \left[\int_{\overline{\mathcal{X}}} 1_{R_x}(z) \; d\Dgm(f)(z)\right] \;dx \\
& = \int_{\overline{\mathcal{X}}} \left[\int_{\R} 1_{R_x}(z) \;dx \right]\; d\Dgm(f)(z) \,.
\end{align*}
For any $z = (z_1,z_2) \in \overline{\mathcal{X}}$,
\be
1_{R_x}(z_1,z_2) = 1_{[0,\infty[}(z_2-x) 1_{[0,\infty[}(x-z_1) \,.
\ee
Using the translation invariance of the Lebesgue measure on $\R$, we get 
\be
\int_{\R} 1_{R_x}(z_1,z_2) \;dx = z_2-z_1 \,.
\ee
This is nothing other than the distance $d_{\R^2,\infty}$ from $(z_1,z_2)$ to the diagonal, so
\be
\int_{\overline{\mathcal{X}}} (z_2-z_1) \; d\Dgm(f)(z_1,z_2) = \Pers_1(\Dgm_k(f)) <\infty \,,
\ee
as soon as $\frac{d}{n} <1$ by virtue of corollary \ref{cor:FinitePersp}.
\end{proof}
\begin{corollary}[No stability for $d_p$-stability for $p>1$]
Let $p>1$ and consider $\beta_k : \DD_p \to L^1(\R)$. Then, $\beta_k$ is discontinuous.
\end{corollary}
\begin{remark}
This suggests that $\chi : \DD_p \to L^1(\R)$ is discontinuous as well.
\end{remark}
\begin{proof}
We give two different arguments which prove the statement. First, consider a diagram $D \notin \DD_1$ such that $D \in \DD_p$ and any other diagram $D' \in \DD_1 \cap \DD_p$. Then, $d_p(D,D')<\infty$, but the reverse triangle inequality entails 
\be
\norm{\norm{\beta_k(D)}- \norm{\beta_k(D')}}_{L^1} \leq \norm{\beta_k(D)- \beta_k(D')}_{L^1}  = \infty\,,
\ee 
since the lower bound is infinite. 
Second, suppose that $\beta_k$ is continuous. Then the map $\mathcal{D}_p \to \R$, given by $D \mapsto \norm{\beta_k(D)}_{L^1}$ must be continuous by continuity of the norm. However, by \cite[Proposition 5.1]{Divol_2019} this map cannot be continuous, since 
\be
\int_{\R} 1_{R_x}(z_1,z_2) \;dx = z_2-z_1  
\ee
is linear in the distance to the diagonal, and in particular not $O(\abs{z_2-z_1}^p)$ as would be required for continuity with respect to $d_p$. 
\end{proof}
\begin{remark}
Divol and Lacombe's work \cite[Proposition 5.1]{Divol_2019} also explain why it is necessary -- even within $\DD_1$ -- to regard the Betti and Euler curves as distributions. Indeed, notice that pointwise in $x$, $1_{R_x}$ doesn't decrease to $0$ at the required rate.
\end{remark}

\begin{corollary}[$d_1$-stability of Betti and Euler curves]
Let $D, D' \in \DD_1$, then
\be
\norm{\beta_k(D)-\beta_k(D')}_{L^1(\R)} \leq 2 \,d_1(D,D') \,.
\ee
\end{corollary}
\begin{remark}
This shows that $\beta_k : \mathcal{D}_1 \to L^1(\R)$ is a $2$-Lipschitz representation. 
\end{remark}
\begin{proof}
Consider a transport map $\pi \in \Gamma(D,D')$. Then,
\begin{align*}
\norm{\beta_k(D)- \beta_k(D')}_{L^1(\R)} \leq \int_{\overline{\mathcal{X}}^2} \left[\int_\R \abs{1_{R_x}(z)-1_{R_x}(w)} \;dx \right] \;d\pi(z,w) \,.
\end{align*}
We carefully check that
\be
\int_\R \abs{1_{R_x}(z)-1_{R_x}(w)} \;dx  \leq 2 \,d_{\R^2, \infty}(z,w) \,,
\ee
which entails the result by taking $\pi$ to be an optimal transport for $d_1$.
\end{proof}

\begin{proposition}[Interpolation for optimal transport]
\label{prop:interpolationOT}
Let $0 < p  < q \leq \infty$ and $\theta \in\, ]0,1[$. Define $p_\theta$ by 
\be
\frac{1}{p_\theta} = \frac{\theta}{p} + \frac{1-\theta}{q} \,.
\ee
Then, for $\mu,\nu \in \mathcal{D}_{p}\cap \mathcal{D}_{q}$ 
\begin{align}
d_{p_\theta}(\mu,\nu) &\leq  2^{1-\theta} \; d_{p}^\theta(\mu,\nu) \,(\Pers_q(\mu) +\Pers_q(\nu))^{1-\theta} \\
d_{p_\theta}(\mu,\nu) &\leq  2^{\theta} \; d_q(\mu,\nu)^{1-\theta} \,(\Pers_p(\mu) +\Pers_p(\nu))^{\theta}  \,.
\end{align}
Consequently, if $p\leq r \leq q$, then $\mathcal{D}_{p}\cap \mathcal{D}_{q} \subset \mathcal{D}_{r}$.
\end{proposition}
\begin{proof}
We prove the first inequality, the second one following by analogy, by interverting the roles of $p$ and $q$. Let $\pi$ be an optimal transport for $d_{p}$. Applying H\"older's inequality,
\begin{align*}
d_{p_\theta}(\mu,\nu) &\leq \norm{d_{\R^2,\infty}}_{L^{p_\theta}(\pi)} \leq \norm{d_{\R^2,\infty}}_{L^{p}(\pi)}^{\theta} \norm{d_{\R^2,\infty}}_{L^{q}(\pi)}^{1-\theta} \\
&= d_{p}(\mu,\nu)^\theta \left[\int_{\overline{\mathcal{X}}^2} d_{\R^2,\infty}^q(z,z') \; d\pi(z,z') \right]^{\frac{1-\theta}{q}} \\
&\leq d_{p}(\mu,\nu)^\theta \left[2^q \int_{\overline{\mathcal{X}}^2} d_{\R^2,\infty}^q(z,\Delta) + d_{\R^2,\infty}^q(\Delta,z') \; d\pi(z,z') \right]^{\frac{1-\theta}{q}} \,,
\end{align*}
where the inequality on the last line holds everywhere on the support of $\pi$. This can be shown by defining 
\be
S=\{(z,z') \in \overline{\mathcal{X}}^2 \cap \supp(\pi) \, \vert \, d_{\R^2,\infty}(z,z') > d_{\R^2,\infty}(z,\Delta) + d_{\R^2,\infty}(z',\Delta) \} \,.
\ee
This set $S$ either has null or positive measure. If it has positive measure, then we can modify the transport plan $\pi$ by sending the projections of $S$ to the diagonal, thereby producing a transport plan of strictly inferior cost to that of $\pi$, which is a contradiction. Hence, $S$ is of null measure, so the equality holds over the support of the measure. This entails
\begin{align*}
d_{p_\theta}(\mu,\nu) \leq 2^{1-\theta}\; d_{p}(\mu,\nu)^\theta \,(\Pers_q(\mu)+\Pers_q(\nu))^{1-\theta} \,.
\end{align*}
If $q =\infty$, since $\pi$ is an optimal transport between $\mu$ and $\nu$ and $\mu,\nu \in \mathcal{D}_\infty$, $\pi$ itself must have compact support and the diameter of the support is bounded above by $\Pers_\infty(\mu) \vee \Pers_\infty(\nu)$, so the inequality of the proposition follows.
\end{proof}
\begin{remark}
Taking $0<p<1\leq q$, proposition \ref{prop:interpolationOT} allows us to say that, despite $\beta_k$ and $\chi$ being discontinuous on $\DD_q$, these functionals are continuous on $\DD_p \cap \DD_q$ (with $p<q$) equipped with the $d_q$ metric. As shown in this paper, we may always do this for diagrams stemming from functions in $W^{n,s}(X,\R)$, provided that $n>d$. This reconciles our results with the ones found in \cite{CurryTurner_2018}, where the case of subanalytic functions was studied (in which case $\Dgm(f) \in \DD_0 \cap \DD_\infty$).
\end{remark}

\begin{theorem}[Stability of Euler and Betti curves under perturbations in distribution]
Let $f$ and $g$ be two stochastic processes on a $d$-dimensional compact Riemannian manifold $X$, defined on a common probability space $(\Omega, \mathcal{F},\PP)$ such that almost surely $f,g \in W^{n,s}(X,\R)$ with $\frac{d}{n}<1$. Then, for every $\frac{d}{n}<\al <1$, there exists a constant $C = C_{X,n,s}$ such that, 
\be
W_{1,L^1}(\beta_k(f)_\sharp\PP,\beta_k(g)_\sharp\PP) \leq C \left[\norm{f}_{L^1(\Omega,W^{n,s})}^\al+\norm{g}_{L^1(\Omega,W^{n,s})}^\al\right] W_{1,L^\infty}^{1-\al}(f_\sharp \PP, g_\sharp \PP) \,,
\ee
and a similar inequality holds by replacing the Betti curves with the Euler curves $\chi$.
\end{theorem}
\begin{proof}
Recognizing that $\beta_k$ are Lipschitz representations, it is a simple application of the stochastic stability theorem and its analog for representations, where we set $q = \al$, $p=1$ and $s = \frac{1}{1-\al}$ and $r = \frac{1}{\al}$.
\end{proof}
\begin{remark}
Using the Kantorovich-Rubenstein duality, a lower bound of $W_{1,L^1}$ can be easily established. Recall that \cite[Particular case 5.16]{Villani_2009}
\be
\sup_{\psi \in \Lip_1(L^1(\R),\R)} \expect{\psi(\chi(f))- \psi(\chi(g))} \leq W_{1,L^1}(\chi(f)_\sharp\PP,\chi(g)_\sharp\PP) \,.
\ee
This point of view further justifies our view of $\beta_k$ and $\chi$ as distributions. Indeed, for $\vp \in \DD(\R)$, the map 
\be
\psi : \chi \mapsto \int_\R \chi(x)\vp(x) dx \,,
\ee
is clearly Lipschitz. This distributional view of the Euler and Betti curves is (unsurprisingly) also stable with respect to perturbations in the distributions of the random fields from which they stem from.

A particularly interesting family of Lipschitz functionals (in light of recent developments \cite{Lebovici_2021}) to consider is
\be
\psi_\theta : \chi \mapsto \int_\R e^{-ix\theta}\chi(x) \;dx \,. 
\ee
Using the Kantorovich-Rubenstein duality on the real and imaginary components of the above integral separately, and taking the supremum over all $\theta$ we arrive that the conclusion that
\be
\norm{\expect{\mathcal{F}\chi(f)-\mathcal{F}\chi(g))}}_{L^\infty(\R)} \leq \sqrt{2} \,W_{1,L^1}(\chi(f)_\sharp\PP,\chi(g)_\sharp\PP) \,,
\ee
where $\mathcal{F}$ denotes the Fourier transform. Interestingly, one can apply the same reasoning for different integral transforms, provided that their kernel is such that the transform is Lipschitz on $L^1$. Of course, these arguments are all also valid for the Betti curves as well. Applying the stability of the Euler and Betti curves under perturbations in the distributions of $f$ and $g$, we arrive at the conclusion that these transforms are stable with respect to changes in distribution.
\end{remark}

\begin{remark}
Looking at Euler curves, it is possible to recover the Euler characteristic of level sets of $f$. Since for every $\veps>0$, $\{f > x-\veps\} \cap \{f < x+\veps\} = \{x-\veps<f<x+\veps\}$ and that $\{f > x-\veps\} \cup \{f < x+\veps\}  = X$, by the valuation property of the Euler characteristic, we recover
\be
\chi(\{x-\veps<f<x+\veps\}) + \chi(X) = \chi(f,x-\veps)+\chi(-f,-x-\veps) \,,
\ee
by taking $\veps \to 0$, we obtain $\chi(f=x)$, whenever it is defined.
\end{remark}

\bibliographystyle{abbrv}
\bibliography{PhDThesis}

\end{document}

%% file: commands.tex
\definecolor{lightgrey}{gray}{0.9}

\newcommand{\expect}[1]{\mathbb{E} \!\left [ #1\right]}

\newcommand{\abs}[1]{\left\lvert #1 \right\rvert}

\newcommand{\innermid}{\nonscript\;\delimsize\vert\nonscript\;}
\newcommand{\activatebar}{%
  \begingroup\lccode`\~=`\|
  \lowercase{\endgroup\let~}\innermid 
  \mathcode`|=\string"8000
}

\newcommand{\C}{\mathbb{C}}

\newcommand{\R}{\mathbb{R}}

\newcommand{\PP}{\mathbb{P}}

\newcommand{\suchthat}{\;\ifnum\currentgrouptype=16 \middle\fi|\;}

\DeclareMathOperator\supp{supp}

\DeclareMathOperator\Pers{Pers}

\DeclareMathOperator\Lip{Lip}

\DeclareMathOperator\Dgm{Dgm}

\newcommand{\al}{\alpha}

\newcommand{\veps}{\varepsilon}
\newcommand{\vp}{\varphi}

\newcommand{\bd}{\begin{displaymath}
\begin{tikzcd}}
\newcommand{\ed}{\end{tikzcd}
\end{displaymath}}
\newcommand{\bmat}{\begin{pmatrix}}
\newcommand{\emat}{\end{pmatrix}}
\newcommand{\be}{\begin{equation}}
\newcommand{\ee}{\end{equation}}
\newcommand{\btikz}{\begin{tikzcd}}
\newcommand{\etikz}{\end{tikzcd}}
\newcommand{\bea}{\begin{eqnarray}}
\newcommand{\eea}{\end{eqnarray}}
\newcommand{\bse}{\begin{subequations}}
\newcommand{\ese}{\end{subequations}}
\newcommand{\bc}{\begin{center}}
\newcommand{\ec}{\end{center}}

\newcommand{\half}{\frac{1}{2}}

\newcommand{\norm}[1]{\left\lVert#1\right\rVert}

\newcommand{\DD}{\mathcal{D}}

\newcommand{\comment}[1]{}

\newcommand{\cf}{{\it cf. }}

\newcommand{\Lag}{{\mathcal{L}}}

\def\blob[#1]{~\parbox{#1mm}{
\begin{fmfgraph*}(#1,#1)
\fmfleft{i1}
\fmfright{o1}
\fmf{phantom}{i1,v1,o1}
\fmfblob{0.4w}{v1}
\end{fmfgraph*}}~}
\def\vertex[#1]{~\parbox{#1mm}{
  \begin{fmfgraph*}(#1,#1)
    \fmfleft{i1, i2}
    \fmfright{o1,o2}
    \fmf{phantom}{i1,i2,o2,o1,i1}
    \fmf{plain}{i1,v,i2}
    \fmf{plain}{o1,v,o2}
    \fmfforce{nw}{i1}
    \fmfforce{sw}{i2}
    \fmfforce{se}{o1}
    \fmfforce{ne}{o2}
    \fmfforce{c}{v}
    \fmfdot{v}
  \end{fmfgraph*}
  }~}
\def\othervertex[#1]{~\parbox{#1mm}{
  \begin{fmfgraph*}(#1,#1)
    \fmfleft{i1, i2}
    \fmfright{o1,o2}
    \fmf{phantom}{i1,i2,o2,o1,i1}
    \fmf{plain}{i1,v}
    \fmf{plain}{i2,v}
    \fmf{plain}{o1,v}
    \fmf{plain}{v,o2}
    \fmfforce{nw}{i1}
    \fmfforce{sw}{i2}
    \fmfforce{se}{o1}
    \fmfforce{ne}{o2}
    \fmfforce{c}{v}
    \fmfblob{0.3w}{v}
  \end{fmfgraph*}
  }~}
\def\bigvertex[#1]{~\parbox{#1mm}{
  \begin{fmfgraph*}(#1,#1)
    \fmfleft{i1, i2}
    \fmfright{o1,o2}
    \fmf{phantom}{i1,i2,o2,o1,i1}
    \fmf{plain}{i1,v,i2}
    \fmf{plain}{o1,v,o2}
    \fmfforce{nw}{i1}
    \fmfforce{sw}{i2}
    \fmfforce{se}{o1}
    \fmfforce{ne}{o2}
    \fmfforce{c}{v}
    \fmfdot{v}
  \end{fmfgraph*}
  }~}
\def\edge[#1]{~\parbox{#1mm}{
  \begin{fmfgraph*}(#1,#1)
    \fmfleft{i}
    \fmfright{o}
    \fmf{plain,l.s=left}{i,o}
  \end{fmfgraph*}
}~}
\def\curlyedge[#1]{~\parbox{#1mm}{
  \begin{fmfgraph*}(#1,#1)
    \fmfleft{i}
    \fmfright{o}
    \fmf{curly,l.s=left}{i,o}
  \end{fmfgraph*}
}~}
\def\wavyedge[#1]{~\parbox{#1mm}{
  \begin{fmfgraph*}(#1,#1)
    \fmfleft{i}
    \fmfright{o}
    \fmf{wiggly,l.s=left}{i,o}
  \end{fmfgraph*}
}~}
\def\tadpole[#1]{~\parbox{#1mm}{
 	\begin{fmfgraph*}(#1,#1)
 		\fmfleft{i1}
 		\fmfright{o1}
		\fmf{plain}{i1,v1,v1,o1}
 		\fmfdot{v1}
 	\end{fmfgraph*}}~}
\def\amputatedtadpole[#1]{\begin{fmfgraph*}(#1,2)
		\fmfleft{i1}
		\fmfright{o1}
		\fmf{phantom}{i1,v1,o1}
		\fmf{plain}{v1,v1}
		\fmfdot{v1}
	\end{fmfgraph*}}
\def\amputatedwigglytadpole[#1]{\begin{fmfgraph*}(#1,2)
		\fmfleft{i1}
		\fmfright{o1}
		\fmf{phantom}{i1,v1,o1}
		\fmf{wiggly}{v1,v1}
		\fmfdot{v1}
	\end{fmfgraph*}}
\def\amputatedcurlytadpole[#1]{\begin{fmfgraph*}(#1,2)
		\fmfleft{i1}
		\fmfright{o1}
		\fmf{phantom}{i1,v1,o1}
		\fmf{curly}{v1,v1}
		\fmfdot{v1}
	\end{fmfgraph*}}
\def\vacfirstord[#1]{~\parbox{#1mm}{
 	\begin{fmfgraph*}(#1,#1)
		\fmfleft{i1,i2}
		\fmfright{o1,o2}
		\fmf{plain,left}{v2,v1}
		\fmf{plain,left}{v1,v2}
		\fmf{plain,left}{v3,v1}
		\fmf{plain,left}{v1,v3}
		\fmf{phantom}{v2,i1}
		\fmf{phantom}{v3,o1}
		\fmfforce{(0,0)}{i1}
		\fmfforce{(w,0)}{o1}
		\fmfforce{(0.5w,0.5h)}{v1}
		\fmfforce{(0,0.5h)}{v2}
		\fmfforce{(w,0.5h)}{v3}
		\fmfforce{nw}{i2}
		\fmfforce{ne}{o2}
		\fmfforce{(0,0.5h)}{i2}
		\fmfforce{(w,0.5h)}{o2}
		\fmfdot{v1}
		\end{fmfgraph*}}~}
\def\doubletadpolehor[#1]{~\parbox{#1mm}{
\begin{fmfgraph}(#1,#1)
	\fmfleft{i1}
	\fmfright{o1}
	\fmf{plain}{i1,v1,v1,v2,v2,o1}
	\fmfdot{v1,v2}
\end{fmfgraph}
}~}
\def\tripletadpolehor[#1]{~\parbox{#1mm}{
\begin{fmfgraph}(#1,#1)
	\fmfleft{i1}
	\fmfright{o1}
	\fmf{plain}{i1,v1,v1,v2,v2,v3,v3,o1}
	\fmfdot{v1,v2,v3}
\end{fmfgraph}
}~}
\def\doubletadpolever[#1]{~\parbox{#1mm}{
\begin{fmfgraph}(#1,#1)
	\fmfleft{i1}
	\fmfright{o1}
	\fmf{plain}{i1,v1}
	\fmf{plain,left}{v1,v2}
	\fmf{plain,left}{v2,v1}
	\fmf{plain}{v1,o1}
	\fmffreeze
	\fmf{plain,left}{v2,v3,v2}
	\fmfforce{c}{v2}
	\fmfforce{(0.5w,h)}{v3}
	\fmfforce{(0.5w,0)}{v1}
	\fmfforce{sw}{i1}
	\fmfforce{se}{o1}
	\fmfdot{v1,v2}
\end{fmfgraph}
}~}
\def\amputateddoubletadpolever[#1]{~\parbox{#1mm}{
	\begin{fmfgraph}(#1,#1)
	\fmfleft{i1}
	\fmfright{o1}
	\fmf{phantom}{i1,v1}
	\fmf{plain,left}{v1,v2}
	\fmf{plain,left}{v2,v1}
	\fmf{phantom}{v1,o1}
	\fmffreeze
	\fmf{plain,left}{v2,v3,v2}
	\fmfforce{c}{v2}
	\fmfforce{(0.5w,h)}{v3}
	\fmfforce{(0.5w,0)}{v1}
	\fmfforce{sw}{i1}
	\fmfforce{se}{o1}
	\fmfdot{v1,v2}
	\end{fmfgraph}~
}
}
\def\sunset[#1]{\parbox{#1mm}{
\begin{fmfgraph}(#1,#1)
\fmfleft{i}
\fmfright{o}
\fmfforce{0,0.5h}{i}
\fmfforce{w,0.5h}{o}
\fmf{plain,tension=5}{i,v1}
\fmf{plain,tension=5}{v2,o}
\fmf{plain,left,tension=0.5}{v1,v2,v1}
\fmf{plain}{v1,v2}
\fmfdot{v1,v2}
\end{fmfgraph}
}}
\def\amputatedsunset[#1]{\parbox{#1mm}{
\begin{fmfgraph}(#1,#1)
\fmfleft{i}
\fmfright{o}
\fmf{phantom,tension=5}{i,v1}
\fmf{phantom,tension=5}{v2,o}
\fmf{plain,left,tension=0.5}{v1,v2,v1}
\fmf{plain}{v1,v2}
\fmfdot{v1,v2}
\end{fmfgraph}
}}
\def\fourpointsecondorder[#1]{~\parbox{#1mm}{
	\begin{fmfgraph}(15,#1)
	\fmfleft{i1,i2}
	\fmfright{o1,o2}
	\fmf{plain}{i1,v1}
	\fmf{plain}{i2,v1}
	\fmf{plain,right}{v1,v2}
	\fmf{plain,left}{v1,v2}
	\fmf{plain}{v2,o1}
	\fmf{plain}{v2,o2}
	\fmfdot{v1,v2}
	\end{fmfgraph}}~
}
\def\fourpointsecondordertwo[#1]{~\parbox{#1mm}{
	\begin{fmfgraph}(15,#1)
	\fmfleft{i1,i2,i3}
	\fmfright{o1}
	\fmf{plain}{i1,v1}
	\fmf{plain}{i2,v1}
	\fmf{plain}{i3,v1}
	\fmf{plain}{v1,v2,v2,o1}
	\fmfdot{v1,v2}
	\end{fmfgraph}
}~}

\newtheorem{theorem}{Theorem}[section]

\catcode`,\active

\catcode`\,12

\theoremstyle{definition}
\newtheorem{definition}[theorem]{Definition}

\newtheorem{convention}[theorem]{Convention}
\newtheorem{conjecture}[theorem]{Conjecture}
\newtheorem{notation}[theorem]{Notation}
\newtheorem{lemma}[theorem]{Lemma}
\newtheorem{proposition}[theorem]{Proposition}

\theoremstyle{remark}
\newtheorem{remark}[theorem]{Remark}
\theoremstyle{example}

\newtheorem{corollary}[theorem]{Corollary}